\documentclass[11pt]{article}

\usepackage{amsmath}
\usepackage{amsthm}
\usepackage{amsfonts}
\usepackage{bbm}

\usepackage[usenames]{color}

\newcommand{\me}{\mathbb{E}}
\newcommand{\E}{\mathbb{E}}
\newcommand{\mr}{\mathbb{R}}
\newcommand{\R}{\mathbb{R}}
\newcommand{\eee}{{\rm e}}
\newcommand{\ind}{\mathbbm{1}}
\newcommand{\dd}{{\rm d}}

\newcommand{\mn}{\mathbb{N}}
\newcommand{\N}{\mathbb{N}}
\newcommand{\mmp}{\mathbb{P}}

\newcommand{\rr}{\rm r}
\DeclareMathOperator{\1}{\mathbbm{1}}

\newcommand{\ii}{{\rm{i}}}

\newtheorem{thm}{Theorem}[section]
\newtheorem{lemma}[thm]{Lemma}

\newtheorem{cor}[thm]{Corollary}

\newtheorem{assertion}[thm]{Proposition}
\theoremstyle{definition}

\theoremstyle{remark}
\newtheorem{rem}[thm]{Remark}

\begin{document}

\title{A functional limit theorem for nested Karlin's occupancy scheme generated by discrete Weibull-like distributions}\date{}
\author{Alexander Iksanov\footnote{Faculty of Computer Science and Cybernetics, Taras Shevchenko National University of Kyiv, Ukraine; e-mail address:
iksan@univ.kiev.ua} \ \ Zakhar Kabluchko\footnote{Institut f\"{u}r Mathematische Statistik, Westf\"{a}lische Wilhelms-Universit\"{a}t M\"{u}nster,
48149 M\"{u}nster, Germany; e-mail address: zakhar.kabluchko@uni-muenster.de} \ \ and \ \ Valeriya Kotelnikova\footnote{Faculty of Computer Science and Cybernetics, Taras Shevchenko National University of Kyiv, Ukraine; e-mail address:
valeria.kotelnikova@unicyb.kiev.ua}}

\maketitle
\begin{abstract}
\noindent Let $(p_k)_{k\in\mathbb{N}}$ be a discrete probability distribution for which the counting function $x\mapsto \#\{k\in\mathbb{N}: p_k\geq 1/x\}$ belongs to the de Haan class $\Pi$. Consider a deterministic weighted branching process generated by $(p_k)_{k\in\mathbb{N}}$. A nested Karlin's occupancy scheme is the sequence of Karlin balls-in-boxes schemes in which boxes of the $j$th level, $j=1,2,\ldots$ are identified with the $j$th generation individuals and the hitting
probabilities of boxes are identified with the corresponding weights. The collection of balls is the same for all generations, and each ball starts at the
root and moves along the tree of the deterministic weighted branching process according to the following rule: transition from a mother box to a daughter box occurs
with probability given by the ratio of the daughter and mother weights.

Assuming there are $n$ balls, denote by $\mathcal{K}_n(j)$ the number of occupied (ever hit) boxes in the $j$th level. For each $j\in\mathbb{N}$, we prove a functional limit theorem for the vector-valued process $(\mathcal{K}^{(1)}_{\lfloor e^{T+u}\rfloor},\ldots, \mathcal{K}^{(j)}_{\lfloor e^{T+u}\rfloor})_{u\in\mathbb{R}}$, properly normalized and centered, as $T\to\infty$. The limit is a vector-valued process whose components are independent stationary Gaussian processes. An integral representation of the limit process is obtained.
\end{abstract}

\noindent Key words: de Haan's class $\Pi$; functional limit theorem; infinite occupancy; nested hierarchy; random environment; stationary Gaussian process

\noindent 2020 Mathematics Subject Classification: Primary: 60F17 
\hphantom{2020 Mathematics Subject Classification: } Secondary: 60G15

\section{Introduction}

\subsection{Definition of the model}

Let $(p_k)_{k\in\mn}$ be a probability distribution, that is, $p_k\geq 0$ for all $k\in\mn$ and $\sum_{k\geq 1}p_k=1$. Additionally, we assume that $p_k>0$ for infinitely many $k$. Also, denote by $(\pi(t))_{t\geq 0}$ a Poisson process on $[0,\infty)$ of unit intensity. Let $S_1, S_2,\ldots$ denote its arrival times, that is,
\begin{equation}\label{eq:1}
\pi(t)=\#\{k\in\mn: S_k\leq t\},\quad t\geq 0.
\end{equation}
In the classical {\it Karlin occupancy scheme} balls are thrown independently into an infinite array of boxes $1,2,\ldots$ with probability $p_k$ of hitting box $k$. There are two standard versions of the Karlin scheme. In the first one that we shall call {\it deterministic} there are $n$ balls thrown, in the second that we shall call {\it Poissonized} the number of balls thrown is $\pi(t)$. The typical question arising in this setting is: what is the asymptotic behavior of various random sequences or functions defined by the scheme as $n\to\infty$ or $t\to\infty$. 

Denote by $\pi_k(t)$ the number of balls which fall into the box $k$ in the Poissonized scheme, so that $\pi(t)=\sum_{k\geq 1}\pi_k(t)$ for all $t\geq 0$. By the known thinning property of Poisson processes, the processes $(\pi_1(t))_{t\geq 0}$, $(\pi_2(t))_{t\geq 0},\ldots$ are independent, and $(\pi_k(t))_{t\geq 0}$ is a Poisson process of intensity $p_k$. Thus, the Poissonized scheme is more tractable than the deterministic scheme, for the numbers of balls falling into different boxes are independent in the former, whereas it is not the case in the latter. This explains a common approach used in most of the papers dealing with the deterministic scheme. First, the scheme is Poissonized. Second, the problem at hand is solved for the Poissonized scheme. Third, the Poissonized scheme is de-Poissonized, that is, transfer is made of the results obtained in the Poissonized scheme to the original deterministic scheme.

In this paper we are interested in a nested family of Karlin's occupancy schemes or simply {\it nested Karlin's occupancy scheme} generated by $(p_k)_{k\in\mn}$ which is defined as follows. Let $\mathcal{R}=\cup_{n\in\mn_0}\mn^n$ be the set of all possible individuals of some population, where $\mn_0:=\mn\cup\{0\}$. The ancestor is identified with
the empty word $\oslash$ and its weight is $p_\oslash=1$. An individual ${\rr}=r_1\ldots r_j$ of the $j$th generation whose weight is denoted by
$p_{\rr}$ produces an infinite number of offspring residing in the $(j+1)$th generation. The
offspring of the individual ${\rr}$ are enumerated by  ${\rr}i =r_1\ldots r_j i$, where $i\in \mn$, and the weights of the offspring are
denoted by $p_{\rr i}$. It is postulated that $p_{\rr i}:=p_{\rr}p_i$. Observe that, for each $j\in\mn$, $\sum_{|{\rr}|=j}p_{\rr}=1$, where, by convention, $|{\rr}|=j$ means that the sum is taken over all individuals of the $j$th generation. We identify individuals with boxes, so that the weights (probabilities) of boxes in the subsequent generations are formed by the vectors $(p_{\rr})_{|{\rr}|=1}=(p_k)_{k\in\mn}$, $(p_{\rr})_{|{\rr}|=2},\ldots$. At time $0$, infinitely many balls are collected in the box $\oslash$. For $n\in\mn$, at the time $n$ in the deterministic scheme or the time $S_n$ in the Poissonized scheme, a new ball arrives and falls, independently of the $(n-1)$ balls that have arrived earlier, into the box ${\rr}$ of the first generation with probability $p_{\rr}=p_r$, and {\it simultaneously} into the box ${\rr}i_1$ of the second generation with probability $p_{i_1}$, into the box ${\rr}i_1i_2$ of the third generation with probability $p_{i_2}$ and so on, indefinitely. A box is deemed {\it occupied} provided it was hit by a ball on its way over the generations. Observe that restricting attention to the $j$th generation we obtain the Karlin occupancy scheme with probabilities $(p_{\rr})_{|{\rr}|=j}$.

For $j\in\mn$, $n\in\mn$ and $t\geq 0$, denote by $\mathcal{K}_n^{(j)}$ and $K_t^{(j)}$ the number of occupied boxes in the $j$th generation when $n$ or $\pi(t)$ balls have been thrown, respectively. Assuming that the probabilities $p_k$ exhibit subexponential (Weibull-like) decay specified by condition \eqref{eq:6} we shall prove weak convergence of the infinite vectors $(\mathcal{K}_n^{(1)}, \mathcal{K}_n^{(2)},\ldots)$ and $(K_t^{(1)}, K_t^{(2)},\ldots)$, properly normalized and centered, as $n\to\infty$ and $t\to\infty$, respectively. % Likewise, with $t$ replacing $n$.

\subsection{Main results}

As usual, we write $\overset{\mmp}{\to}$ to denote convergence in probability, and $\Rightarrow$, ${\overset{{\rm
d}}\longrightarrow}$ and ${\overset{{\rm f.d.d.}}\longrightarrow}$
to denote weak convergence in a function space, weak convergence
of one-dimensional and finite-dimensional distributions,
respectively. Also, we denote by $D:=D(\mr)$ the Skorokhod space of right-continuous functions defined on
$\mr$ with finite limits from the left.

Put $\rho(x):=\#\{k\in\mn: p_k\geq 1/x\}$ for $x>0$ and note that $\rho(x)=0$ for $x\in (0, 1]$ (unless $p_k=1$ for some $k\in\mn$). For each $T\geq 0$ and $j\in\mn$, put $${\bf K}^{(j)}(T,u):=\frac{K^{(j)}_{\eee^{T+u}}-\me K^{(j)}_{\eee^{T+u}}}{({\rm Var}\,K^{(j)}_{\eee^T})^{1/2}},\quad u\in\mr$$ and
$${\bf \mathcal{K}}^{(j)}(T,u):=\frac{\mathcal{K}^{(j)}_{\lfloor \eee^{T+u}\rfloor }-\me \mathcal{K}^{(j)}_{\lfloor \eee^{T+u}\rfloor}}{({\rm Var}\,\mathcal{K}^{(j)}_{\lfloor \eee^T\rfloor })^{1/2}},\quad u\in\mr.$$ Note that ${\bf K}^{(j)}(T):=({\bf K}^{(j)}(T,u))_{u\in\mr}$ and ${\bf \mathcal{K}}^{(j)}(T):=({\bf \mathcal{K}}^{(j)}(T,u))_{u\in\mr}$ are random elements with values in $D$. Here are our main results, Theorem \ref{main} for the Poissonized scheme and Corollary \ref{main1} for the deterministic scheme.
\begin{thm}\label{main}
Assume that, for all $\lambda>0$,
\begin{equation}\label{eq:6}
\lim_{t\to\infty}\frac{\rho(\lambda t)-\rho(t)}{(\log t)^\beta\ell(\log t)}=\log\lambda
\end{equation}
for some $\beta\geq 0$ and some $\ell$ slowly varying at $\infty$. If $\beta=0$ we further assume that $\ell$ is eventually nondecreasing and unbounded. Then
\begin{equation}\label{relation_main5000}
\big({\bf K}^{(j)}(T)\big)_{j\geq 1}~\Rightarrow~(Z_j)_{j\geq 1},\quad T\to\infty
\end{equation}
in the product $J_1$-topology on $D^\mn$. Here, $Z_1$, $Z_2,\ldots$ are independent copies of a centered stationary Gaussian process $Z:=(Z(u))_{u\in\mr}$ with covariance $$\me Z(u)Z(v)=\frac{\log \big(1+\eee^{-|u-v|}\big)}{\log 2},\quad u,v\in\mr.$$
\end{thm}
\begin{cor}\label{main1}
Under the assumptions of Theorem \ref{main},
\begin{equation}\label{relation_main50001}
\big({\bf \mathcal{K}}^{(j)}(T)\big)_{j\geq 1}~\Rightarrow~(Z_j)_{j\geq 1},\quad T\to\infty
\end{equation}
in the product $J_1$-topology on $D^\mn$.
\end{cor}

\begin{rem}
A typical example in which condition \eqref{eq:6} holds is given by Weibull-like probabilities $p_k=C\exp(-k^\alpha)$, $k\in\mn$, where $\alpha\in (0,1)$. In this case, $\rho(x)=\lfloor (\log (Cx))^{1/\alpha}\rfloor$, where $\lfloor a\rfloor$ denotes the integer part of real $a$, so that condition \eqref{eq:6} holds with $\beta=\alpha^{-1}-1$ and $\ell(t)=\alpha^{-1}$ for $t>0$. A counterpart of relation \eqref{eq:6} expressed in terms of $p_k$ can be found in Proposition \ref{prop:suff}.
\end{rem}
\begin{rem}
A function $h:(0,\infty)\to\mr$ is said to belong to de Haan's class $\Pi$ with the auxiliary function $g$ if, for all $\lambda>0$,
\begin{equation}\label{eq:dehaan}
\lim_{t\to\infty}\frac{h(\lambda t)-h(t)}{g(t)}=\log \lambda,
\end{equation}
and $g$ is slowly varying at $\infty$. See Section 3 in \cite{Bingham+Goldie+Teugels:1989} for detailed information about class $\Pi$.

Not only does condition \eqref{eq:6} tell us that $\rho$ belongs to the class $\Pi$, but also puts restrictions on the auxiliary function. The explicit form of the auxiliary function is essentially used in the proof of Proposition \ref{prop:regular}, when showing that the counting functions of probabilities in generations $2$, $3,\ldots$ belong to the class $\Pi$, as well. Apart from this, the extra information provided by the explicit form is not needed. For instance, weak convergence of the first coordinate in \eqref{relation_main5000} could have been proved under the sole assumption \eqref{eq:dehaan} with $\rho$ replacing $h$.
\end{rem}

In Theorem \ref{thm:main2} we provide an integral representation of the limit process $Z$ and prove that there is a version of $Z$ with continuous sample paths.
\begin{thm}\label{thm:main2}
Let $W(\dd x, \dd y)$ be a white noise on the horizontal strip $\R\times [0,\,1]$ whose intensity measure is the standard Lebesgue measure. Put
$$Z(u) := \frac{1}{(\log 2)^{1/2}} \int_{\R\times [0,\,1]} \left( \ind_{\left\{y \leq \exp(-\eee^{-(x-u)})\right\}} -\exp(-\eee^{-(x-u)})\right)W(\dd x, \dd y),\qquad u\in \R.$$ Then $Z:=(Z(u))_{u\in \R}$ is a stationary centered Gaussian process whose covariance function is given by
$${\rm Cov}\,(Z(u), Z(v)) = \frac{\log (1 + \eee^{- |u-v|})}{\log 2},
\qquad u,v\in \R.$$ Furthermore, the process $Z$ has a version with sample paths which are Hölder continuous with exponent $\alpha$, for every $\alpha\in (0,1/2)$.
\end{thm}

\section{Relevant literature}

For $n\in\mn$ and $i\leq n$, denote by $\mathcal{K}_n(i)$, $\mathcal{K}_n:=\mathcal{K}^{(1)}_n$ and $\mathcal{K}^{{\rm odd}}_n$ the number of boxes containing at least $i$ balls, the number of occupied boxes and the number of boxes containing an odd number of balls, respectively, when $n$ balls have been thrown. We start by giving a brief review of articles dealing with one-dimensional and functional central limit theorems for the just introduced (and related) quantities arising in Karlin's occupancy scheme. Some other aspects of the model are discussed in the survey \cite{Gnedin+Hansen+Pitman:2007}.

As far as we know, there are only few articles in which functional limit theorems for Karlin's occupancy scheme were proved. We think it is quite a surprising fact in view of almost 50 years long history of the model. Under the assumption
\begin{equation}\label{eq:regular}
\rho(x)~\sim x^\alpha \ell(x),\quad x\to\infty
\end{equation}
for some $\alpha\in (0,1)$ and some $\ell$ slowly varying at $\infty$, in \cite{Durieu+Wang:2016} functional central limit theorems for the processes $(\mathcal{K}_{\lfloor nt\rfloor})_{t\in [0,\,1]}$ and $(\mathcal{K}^{{\rm odd}}_{\lfloor nt\rfloor})_{t\in [0,\,1]}$, properly normalized and centered, as $n\to\infty$ were obtained. The weak limit for each process is a fractional Brownian motion. These results complement the one-dimensional central limit theorems for $\mathcal{K}_n$ and $\mathcal{K}_n^{{\rm odd}}$ due to Karlin (Theorems 4 and 6 in \cite{Karlin:1967}). Note that an ultimate version of the one-dimensional central limit theorem for $\mathcal{K}_n$ was given in \cite{Dutko:1989} under the sole assumption that $\lim_{n\to\infty}{\rm Var}\,\mathcal{K}_n=\infty$. Note that the last limit relation may hold even if condition \eqref{eq:regular} fails to hold. Theorem 2.1 in \cite{Hwang+Janson:2008} provides a local central limit theorem for $\mathcal{K}_n$. In \cite{Durieu+Wang:2016}, see also \cite{Durieu+Samorodnitsky+Wang:2020}, functional central limit theorems were proved for certain randomized versions of $(\mathcal{K}_{\lfloor nt\rfloor})_{t\in [0,1]}$ and $(\mathcal{K}^{{\rm odd}}_{\lfloor nt\rfloor})_{t\in [0,1]}$. In \cite{Chebunin+Kovalevskii:2016}, under \eqref{eq:regular}, a functional limit theorem for the process $(\mathcal{K}_{\lfloor nt\rfloor}(1),\ldots, \mathcal{K}_{\lfloor nt\rfloor}(i))_{t\in [0,\,1]}$ ($i\in\mn$), properly normalized and centered, was obtained. The limit is an $i$-dimensional self-similar Gaussian process. A recent extension of this theorem which particularly covers the case $\alpha=1$ in \eqref{eq:regular} can be found in \cite{Chebunin+Zuyev:2021}.

When the condition
\begin{equation}\label{eq:regular0}
\rho(x)~\sim~ \ell(x),\quad x\to\infty
\end{equation}
holds which is the situation we are focussed at, much less was known. In particular, to the best of our knowledge, under \eqref{eq:regular0}, functional limit theorems for $\mathcal{K}_n$ or $K_t$, properly scaled, centered and normalized, have not been proved so far. Under the assumption that the function $\sum_{k\geq 1}p_k\1_{\{p_k\leq 1/x\}}$ is regularly varying at $\infty$ of index $-1$ which entails \eqref{eq:regular0} it was proved on p.~380 in \cite{Barbour+Gnedin:2009} that the Poissonized version of $(\mathcal{K}_n(1)-\mathcal{K}_n(2), \ldots, \mathcal{K}_n(i)-\mathcal{K}_n(i-1))$ ($i\in\mn$), properly normalized and centered, converges in distribution to an $i$-dimensional Gaussian vector. On the other hand, there is a huge literature addressing various aspects of the Karlin occupancy scheme in the rather particular geometric case $p_k=p(1-p)^{k-1}$, $k\in\mn$. This interest is partly motivated by a connection to the {\it leader election problem} and mathematical tractability which results from possibility of explicit (but tedious) calculations. We refrain from giving a survey, a selection of relevant articles can be traced via the references given in Section 1.2 of \cite{Bogachev+Gnedin+Yakubovich:2008}. We only mention that the Poissonized version of the number of occupied boxes centered by its mean (without normalization) only converges in distribution along subsequences. This explains the fact that condition \eqref{eq:6} is not satisfied by the geometric distribution. We note in passing that no normalization and centering for $\mathcal{K}_n^{(j)}$ exists which would ensure convergence in distribution. This implies that condition \eqref{eq:5} given below can hold for no $j$, the fact that can alternatively be checked by a direct calculation.

There is a version of the nested Karlin's occupancy scheme, called {\it nested occupancy scheme in random environment}, in which the distribution $(p_k)_{k\in\mn}$ is random. Such a model was introduced in \cite{Bertoin:2008} and further investigated in \cite{Buraczewski+Dovgay+Iksanov:2020}, \cite{Businger:2017}, \cite{Gnedin+Iksanov:2020}, \cite{Iksanov+Marynych+Samoilenko:2020}, \cite{Joseph:2011}. In \cite{Bertoin:2008} and \cite{Joseph:2011} the asymptotics of the number of occupied boxes $K_n^{(j)}$ and related quantities was analyzed at the generations $j$ of order $\log n$. Some results of the last two cited papers apply to the nested Karlin occupancy scheme. We are not aware of any articles which would treat the generations $j$ with $j=j_n\to\infty$ and  $j_n=o(\log n)$ as $n\to\infty$ of the nested Karlin occupancy scheme.

\section{Informal derivation of the limit process}

We start with some preparations. Let $j\in\mn$. First we give a representation for $K_t^{(j)}$ to be used for several times. Denote by $T_{\rr}$ the time at which the box $\rr$ with $|{\rr}|=j$ is filled for the first time. Observe that $T_{\rr}$ has an exponential distribution with parameter $p_{\rr}$, that is, $\mmp\{T_{\rr}>t\}=\eee^{-p_{\rr} t}$ for $t\geq 0$, and that the collection $(T_{\rr})_{|{\rr}|=j}$ consists of independent random variables. For the box ${\rr}$ with $|{\rr}|=j$, put $$G_{\rr} := -\log p_{\rr} - \log T_{\rr}.$$
Then the collection $(G_{\rr})_{|{\rr}|=j}$ consists of independent random variables with the standard Gumbel distribution, that is,
$$
\mmp\{G_{\rr} \leq x\} = \exp(-\eee^{-x}), \quad x\in \R.
$$
Assume that the number of balls thrown is $\pi(t)$. Then the number of occupied boxes is given by
\begin{equation}\label{eq:jthgener}
K_t^{(j)} = \sum_{|{\rr}|=j}\1_{\{T_{\rr}\leq t\}} = \sum_{|{\rr}|=j}\1_{\{-\log p_{\rr} - G_{\rr} \leq \log t\}}.
\end{equation}

In the remainder of this section we provide an informal derivation of the limit process $Z$ in Theorem \ref{main}. Also, we prove Theorem \ref{thm:main2} and identify the spectral density of $Z$ in Proposition \ref{prop:spectral}. For simplicity of presentation we only consider the first generation. When $j=1$, we can work with the collection $(p_k, T_k, G_k)_{k\in\mn}$ instead of $(p_{\rr}, T_{\rr}, G_{\rr})_{|{\rr}|=1}$  in which case \eqref{eq:jthgener} reads
\begin{equation}\label{eq:firstgen}
K_t:=K_t^{(1)} = \sum_{k\geq 1} \ind_{\{T_k\leq t\}} = \sum_{k\geq 1}\ind_{\{-\log p_k - G_k \leq \log t\}}.
\end{equation}
Essentially, $K_t$ is an empirical process generated by the deterministic points $-\log p_k$, $k\in \N$ with independent random Gumbel shifts. In the following, we shall approximate $K_t$, properly normalized and centered, by a combination of Brownian bridges.
For $t= \eee^{T+u}$ the last centered formula reads
\begin{equation}\label{eq:K_aux}
K_{\eee^{T+u}}
= \sum_{k\geq 1} \ind_{\{-\log p_k - G_k \leq T+u\}}
=\sum_{k\geq 1} \ind_{\{G_k \geq -\log p_k - T - u\}}
=\sum_{k\geq 1}\left(1 - \ind_{\{G_k<-\log p_k - T - u\}}\right).
\end{equation}
Put $f(t):=t^\beta\ell(t)$ for large $t$, with the same $\beta$ and $\ell$ as in \eqref{eq:6}. According to formula \eqref{eq:414}, ${\rm Var}\,K_{\eee^T}\sim (\log 2) f(T)$ as $T\to\infty$. Hence, we can work with the process ${\bf K}^\ast(T,u)$ defined by 
$${\bf K}^\ast(T,u):=\frac{K_{\eee^{T+u}} - \E K_{\eee^{T+u}}}{\sqrt {f(T)}}=- \sum_{k\geq 1} \frac{\ind_{\{G_k<-\log p_k - T - u\}} - \mmp\{G_k<-\log p_k - T - u\}}{\sqrt{f(T)}}$$ rather than ${\bf K}^{(1)}(T,u)$. We now argue (in an informal way) that, for any $A>0$, the process $({\bf K}^\ast(T,u))_{u\in [-A,\,A]}$ converges weakly as $T\to \infty$ and identify the limit. To this end, take an infinitesimal interval $[x, x+\dd x]$ and consider all those $k$ for which $-\log p_k - T \in [x, x + \dd x]$. The number of such points is
$$\rho(\eee^{T + x + \dd x}) - \rho (\eee^{T + x})~\sim~f(T+x)\dd x~ \sim~ f(T)\dd x,\quad T\to\infty$$ in view of \eqref{eq:6}. The contribution of such $k$'s to ${\bf K}^\ast(T,u)$ is
$$-\sum_{\substack{k\geq 1\\ -\log p_k - T \in [x,\, x + \dd x]}} \frac{\ind_{\{G_k< x - u \}} - \mmp\{G_k<x-u\}}{\sqrt {f(T)}}$$
or, equivalently, $$ -\sum_{\substack{k\geq 1\\ -\log p_k - T \in [x, x + \dd x]}} \frac{\ind_{\{U_k<\exp(-\eee^{-(x - u)})\}} -\exp(-\eee^{-(x-u)})}{\sqrt {f(T)}},$$
where the random variables $U_k :=\exp(-\eee^{-G_k})$, $k\in\N$ are independent and uniformly distributed on $[0,1]$. Recall that the number of points $U_k$ contributing to this sum is approximately $f(T)\dd x$. Since the uniform empirical process converges weakly to a Brownian bridge as $T\to\infty$,  we can approximate the latter sum by a process of the form $\sqrt {\dd x} B_x^0 (\eee^{-\eee^{-(x-u)}})$, where $(B_x^0 (t))_{t\in [0,1]}$ is a Brownian bridge. This approximation makes sense for every small interval of the form $[x, x + \dd x]$, and, moreover, disjoint intervals correspond to  independent Brownian bridges. Thus, as $T\to\infty$, the process $({\bf K}^\ast(T,u))_{u\in\R}$ can be approximated by an ``integral'' of independent infinitesimal contributions of the form $\sqrt {\dd x} B_x^0 (\exp(-\eee^{-(x-u)}))$ taken over all $x\in \R$. Finally, note that the Brownian bridge $B_x^0$ can be written as $B_x^0 (t) = B_x(t) - t B_x(1)$, where $(B_x(t))_{t\geq 0}$ is a Brownian motion. This naturally leads to a stochastic integral representation of the limit process given in Theorem \ref{thm:main2}.
\begin{proof}[Proof of Theorem \ref{thm:main2}]
In what follows it is tacitly assumed that the $x$-variable ranges in $\R$, and the $y$-variable ranges in $[0,1]$.
By the properties of the stochastic integrals with respect to a white noise,
$${\rm Cov}\,(Z(u), Z(v))=\frac 1 {\log 2}\int_{\R\times [0,\,1]} \left( \ind_{\left\{y \leq \eee^{-\eee^{-(x-u)}}\right\}} - \eee^{-\eee^{-(x-u)}} \right)  \left( \ind_{\left\{y \leq \eee^{-\eee^{-(x-v)}}\right\}} - \eee^{-\eee^{-(x-v)}} \right) \dd x \dd y.$$
For $a,b\in [0,1]$, % we have
$$\int_0^1 (\ind_{\{y\leq a\}} -a)(\ind_{\{y<b\}} - b) \dd y = \min (a,b) - ab.$$ Assume without loss of generality that $u\leq v$. Then $\eee^{-\eee^{-(x-u)}}\geq \eee^{-\eee^{-(x-v)}}$. With these at hand, fixing some $x\in \R$ and integrating over $y\in [0,1]$ we obtain
$${\rm Cov}\,(Z(u), Z(v))=\frac 1 {\log 2} \int_{\R}\left( \eee^{-\eee^{-(x-v)}} - \eee^{-\eee^{-(x-u)}} \eee^{-\eee^{-(x-v)}} \right) \dd x.$$
The substitution $t:= \eee^{-x}$ transforms the integral into a Frullani integral
$${\rm Cov}\,(Z(u), Z(v))=\frac 1 {\log 2} \int_0^\infty \left( \eee^{- t \eee^v} - \eee^{- t (\eee^u + \eee^ v)}\right) \frac{\dd t}{t}
=\frac{\log (1 + \eee^{-(v-u)})}{\log 2},$$
which completes the proof of the first claim.

Since, as $u\to 0$,
$$\log \big(1 + \eee^{-|u|}\big) = \log (2  - |u| + o(u)) = \log 2 - \frac 12 |u| + o(u),$$
the second claim is justified by the Kolmogorov-Chentsov theorem.
\end{proof}

Recall that the spectral density of a stationary process with covariance function $r$ is a Lebesgue integrable on $\R$ function $f$ whose Fourier transform is $r$, that is,
$$r(t) = \int_\R \eee^{\ii t x} f(x) \dd x, \qquad t\in \R.$$
\begin{assertion}\label{prop:spectral}
The spectral density of the process $Z$ % under consideration
is given by
$$
f(x) = \frac{1}{2\log 2}\Big(\frac{1}{\pi x^2} - \frac{1}{ x \sinh (\pi x)}\Big), \qquad x\in \R \backslash\{0\}, \quad f(0) = \frac {\pi}{12\log 2}.
$$
\end{assertion}
\begin{proof}
The function $r_1$ defined by $r_1(t) = \log (1+\eee^{-|t|})$ for $t\in\mr$ is continuous. Also, it is
%Consider the function $r(t) = \log (1+\eee^{-|t|})$. It is continuous and
positive definite as the covariance function of a stationary process. Therefore, it is the characteristic function of a finite measure $\mu$ (of total mass $\log 2$). Moreover, the function $r_1$ is integrable which implies that $\mu$ has a density $f_1$ which is given by the inversion formula
$$
f_1(x) = \frac 1 {2\pi} \int_\R \eee^{-\ii t x} r_1(t) \dd t,\quad x\in\mr.
$$
By the Taylor expansion with the Lagrange form of the remainder, % we can write
$$
r_1(t)  = \log (1 + \eee^{-|t|}) = \sum_{n=1}^{N-1} \frac{(-1)^{n-1}}{n} \eee^{- n |t|} + R_N(t),\quad t\in\mr.
$$
Here, %where the remainder term satisfies
$|R_N(t)| \leq \frac 1N \eee^{-N |t|}$ for $t\in \R$, whence, for every $x\in\mr$, $\lim_{N\to\infty} \int_\R \eee^{-\ii t x} R_N(t) \dd t=0$. %  converges to $0$ for every $x\in \R$,
%In particular, $R_N$ converges to $0$ as $N\to\infty$ in $L^1$. Hence, $\int_\R \eee^{- it x} R_N(t) \dd t$  converges to $0$ for every $x\in \R$, as $N\to\infty$, and we can write
As a consequence, $$f_1(x)=\frac 1 {2\pi}\int_\R \eee^{-\ii t x} \sum_{n\geq 1} \frac{(-1)^{n-1}}{n} \eee^{- n |t|} \dd t
= \frac 1 {2\pi} \sum_{n\geq 1} \frac{(-1)^{n-1}}{n}  \int_\R \eee^{-\ii t x}   \eee^{- n |t|} \dd t = \frac 1 {\pi} \sum_{n\geq 1} \frac{(-1)^{n-1}}{x^2+n^2}.$$
Applying the partial fraction expansion $$\frac 1 {\sinh z} - \frac 1 z = 2 z \sum_{n\geq 1}\frac{(-1)^n}{z^2 + \pi^2 n^2},$$ we arrive at the claimed formula for $f=f_1/\log 2$.

%By the Taylor expansion of the logarithm, % we have
%$$
%r(u)  = \log \big(1 + \eee^{-|u|}\big) = \sum_{n\geq 1}\frac{(-1)^{n-1}}{n} \eee^{- n |u|}.
%$$
%According to the formula for the characteristic function of the Cauchy distribution, % we have
%$$
%\eee^{- n |u|} = \int_\R \frac{\eee^{\ii u x}}{\pi n (1 + \frac {x^2}{n^2})}\dd x.
%$$
%Interchanging the sum and the integral yields %, we get
%$$
%r(u) = \int_\R \eee^{\ii u x} \sum_{n\geq 1} \frac{(-1)^{n-1}}{\pi(n^2 + x^2)}\dd x.
%$$
%By the partial fraction expansion $\frac 1 {\sinh z} - \frac 1 z = 2 z \sum_{n \geq 1}\frac{(-1)^n}{z^2 + \pi^2 n^2}$, the sum under the sign of integral is equal to %equals the function
%$f(x)$. % given in the statement.
\end{proof}

\section{Auxiliary results}

By Proposition \ref{prop:regular} given below condition \eqref{eq:6} ensures that
\begin{equation}\label{eq:reg}
\rho(x)~\sim~(\beta+1)^{-1}(\log x)^{\beta+1}\ell(\log x),\quad x\to\infty.
\end{equation}
To facilitate application of Theorem \ref{main} we give in Proposition \ref{prop:suff} sufficient conditions for \eqref{eq:reg} and \eqref{eq:6} expressed in terms of $p_k$.
\begin{assertion}\label{prop:suff}
Assume that the sequence $(p_k)_{k\in\mn}$ is eventually nonincreasing.

\noindent (a) If
\begin{equation}\label{eq:reg1}
-\log p_{\lfloor x\rfloor}~\sim~x^{1/b}L(x^{1/b}),\quad x\to\infty
\end{equation}
for some $b>0$ and some $L$ slowly varying at $\infty$. Then
\begin{equation}\label{eq:reg2}
\rho(x)~\sim~(\log x)^b(L^\#(\log x))^b,\quad x\to\infty,
\end{equation}
where $L^{\#}$ is the de Bruijn conjugate for $L$. In particular, relation \eqref{eq:reg} is secured by \eqref{eq:reg1} with $b=\beta+1$ and $L(x)=\ell_0^\#(x)$, that is,
\begin{equation}\label{eq:reg3}
-\log p_{\lfloor x\rfloor}~\sim~x^{1/(\beta+1)}\ell_0^{\#}(x^{1/(\beta+1)}),\quad x\to\infty.
\end{equation}
Here, $\ell_0(x):=(\ell(x)/(\beta+1))^{1/(\beta+1)}$.

\noindent (b) Assume that \eqref{eq:reg3} holds and that, for all $u\in\mr$,
\begin{equation}\label{eq:reg4}
\lim_{t\to\infty}\frac{s(t+ug(t))}{s(t)}=e^u,
\end{equation}
where $s(x):=1/p_{\lfloor x\rfloor}$ for $x\geq 1$ and $g$ is a nonnegative measurable function satisfying $$g(t)~\sim~(\beta+1)^{\beta/(\beta+1)} t^{\beta/(\beta+1)}(\ell(t^{1/(\beta+1)}))^{1/(\beta+1)},\quad t\to\infty$$ and
\begin{equation}\label{eq:reg5}
\ell(t(\ell(t))^{1/(\beta+1)})~\sim~\ell(t),\quad t\to\infty.
\end{equation}
Then \eqref{eq:6} holds.
\end{assertion}
\begin{rem}
Condition \eqref{eq:reg5} is satisfied, for instance, by powers of logarithms and iterated logarithms. Further, it is satisfied by $\ell(x)=\exp((\log x)^c)$ if $c\in (0,1/2)$ and not satisfied if $c\in [1/2, 1)$, see p.~78 in \cite{Bingham+Goldie+Teugels:1989}. According to Theorem 2.3.3 in \cite{Bingham+Goldie+Teugels:1989}, $$\lim_{t\to\infty}\Big(\frac{\ell(\lambda_0t)}{\ell(t)}-1\Big)\log \ell(t)=0$$ for some $\lambda_0>1$ is a sufficient condition for \eqref{eq:reg5}.
\end{rem}
\begin{proof}[Proof of Proposition \ref{prop:suff}]
(a) Note that, as $x\to\infty$,
\begin{equation}\label{eq:inter1}
\rho(e^x)=\#\{k\geq 1: -\log p_k\leq x\}~\sim~\inf\{k\geq 1: -\log p_k>x\}~\sim~\inf\{y\geq 1: -\log p_{\lfloor y\rfloor}>x\}.
\end{equation}
By Proposition 1.5.15 in \cite{Bingham+Goldie+Teugels:1989}, $$\inf\{y\geq 1: -\log p_{\lfloor y\rfloor}>x\}~\sim~ x^b (L^{\#}(x))^b,\quad x\to\infty$$ which completes the proof of \eqref{eq:reg2}. The second part of (a) follows from the first and the fact that $\ell_0^{\#\#}(x)\sim \ell_0(x)$ as $x\to\infty$, see Theorem 1.5.13 in \cite{Bingham+Goldie+Teugels:1989}.

\noindent (b) We note in passing that since $s$ satisfies \eqref{eq:reg4}, it belongs to the class $\Gamma$, see p.~175 in \cite{Bingham+Goldie+Teugels:1989}.

Put $s^\leftarrow(x):=\inf\{y\geq 1: s(y)>x\}$ for $x>1$. By the argument leading to \eqref{eq:inter1}, the function $\rho-s^\leftarrow$ is bounded. Thus, under the sole assumption \eqref{eq:reg4}, for all $\lambda>0$, $$\lim_{t\to\infty}\frac{\rho(\lambda t)-\rho(t)}{g(\rho(t))}=\log \lambda$$ by Theorem 3.10.4 in \cite{Bingham+Goldie+Teugels:1989}. Thus, it remains to show that, under \eqref{eq:reg3}, $$g(\rho(e^t))~\sim~t^\beta\ell(t),\quad t\to\infty.$$ According to part (a), relation
\eqref{eq:reg} holds. Using this together with the assumptions imposed on $g$ and $\ell$ we obtain
\begin{multline*}
g(\rho(e^t))~\sim~ (\beta+1)^{\beta/(\beta+1)}((\beta+1)^{-1}t^{\beta+1}\ell(t))^{\beta/(\beta+1)}(\ell(t(\ell(t))^{1/(\beta+1)}))^{1/(\beta+1)}\\~\sim~t^\beta (\ell(t))^{\beta/(\beta+1)}(\ell(t))^{1/(\beta+1)} =t^\beta \ell(t),\quad t\to\infty.
\end{multline*}
\end{proof}

We recall that the Euler gamma function $\Gamma$ is defined by $\Gamma(x)=\int_0^\infty y^{x-1}e^{-y}{\rm d}y$ for $x>0$. For $j\in\mn$, put $$\rho_j(x):=\#\{{\rr}: |{\rr}|=j,~~ p_{\rr}\geq 1/x\},\quad x>0.$$ In particular, $\rho_1(x)=\rho(x)$ for $x>0$. In Proposition \ref{prop:regular} we prove that $\rho_j$ satisfies a counterpart of \eqref{eq:6}, thereby showing % which means
that, similarly to $\rho$, it belongs to de Haan's class $\Pi$. Also, we point out the first-order asymptotic behavior of $\rho_j$.
\begin{assertion}\label{prop:regular}
Under the assumptions of Theorem \ref{main}, for integer $j\geq 2$ and $\lambda>0$,
\begin{equation}\label{eq:5}
\lim_{t\to\infty}\frac{\rho_j(\lambda t)-\rho_j(t)}{(\log t)^{j\beta+j-1}(\ell(\log t))^j}=\frac{(\Gamma(\beta+1))^j}{\Gamma(j(\beta+1))}\log\lambda
\end{equation}
and
\begin{equation}\label{eq:7}
\lim_{t\to\infty} \frac{\rho_j(t)}{(\log t)^{j(\beta+1)}(\ell(\log t))^j}=\frac{(\Gamma(\beta+1))^j}{\Gamma(1+j(\beta+1))}.
\end{equation}
\end{assertion}
\begin{proof}
It is more convenient to work with $V_j(t):=\rho_j(e^t)$, $j\in\mn$, $t\geq 0$, % and $g_j(t):=f_j(e^{t})$ for $t\geq 0$,
where $\rho_1:=\rho$. Then \eqref{eq:6} and \eqref{eq:5} can be written in an equivalent form:
for $j\in\mn$ and $h\in\mr$,
\begin{equation}\label{eq:5ref}
\lim_{t\to\infty}\frac{V_j(t+h)-V_j(t)}{f_j(t)}=h,
\end{equation}
where $f_j$ is a nonnegative function satisfying
\begin{equation}\label{eq:5refref}
f_j(t)~\sim~\frac{(\Gamma(\beta+1))^j}{\Gamma(j(\beta+1))}t^{j\beta+j-1}(\ell(t))^j,\quad t\to\infty.
\end{equation}
When $\beta>0$, by Theorem 1.5.3 in \cite{Bingham+Goldie+Teugels:1989}, we can and do assume that $f_1$ is nondecreasing. When $\beta=0$, we can put $f_1=\ell_1$ where $\ell_1$ is a nondecreasing on $[0,\infty)$ modification of $\ell$. Also, in both cases, adjusting if needed $f_1$ at discontinuity points, the number of these being at most countable, we can assume that $f_1$ is right-continuous. The latter property is needed for a proper application of Theorem 1.7.1 from \cite{Bingham+Goldie+Teugels:1989} below.

We use the mathematical induction on $j$. If $j=1$, then \eqref{eq:5ref} and \eqref{eq:5refref} are secured by \eqref{eq:6}. Assume that \eqref{eq:5ref} and \eqref{eq:5refref} hold for $j\leq k$. Then \eqref{eq:5} also holds and entails \eqref{eq:7} with $j\leq k$ by the implication (3.7.6) $\Rightarrow$ (3.7.8) of Theorem 3.7.3 in \cite{Bingham+Goldie+Teugels:1989}. % We shall write $V$ for $V_1$.
In view of \eqref{eq:5ref} with $j=1$, given $\varepsilon\in (0,h)$ there exists $t_0>0$ such that $V(t+h)-V(t)\geq (h-\varepsilon)f(t)$ whenever $t\geq t_0$. Here and hereafter, we write $V$ and $f$ for $V_1$ and $f_1$. Using
\begin{multline*}
V_{k+1}(t)=\sum_{|{\rr}|=k+1}\1_{\{p_{\rr}\geq \eee^{-t}\}}=\sum_{|{\rr}|=k}\sum_{i\geq 1}\1_{\{p_{\rr}p_i\geq \eee^{-t}\}}=\sum_{|{\rr}|=k}V(t-|\log p_{\rr}|)\1_{\{|\log p_{\rr}|\leq t\}}\\=\int_{[0,\,t]}V(t-x){\rm d}V_k(x),
\end{multline*}
we write
\begin{multline*}
V_{k+1}(t+h)-V_{k+1}(t)=\int_{[0,\,t-t_0]}(V(t+h-x)-V(t-x)){\rm d}V_k(x)\\+\int_{(t-t_0,\,t]}(V(t+h-x)-V(t-x)){\rm d}V_k(x)+\int_{(t,\,t+h]}V(t+h-y){\rm d}V_k(y)=:A_k(t)+B_k(t)+C_k(t)
\end{multline*}
and analyze the summands separately. We first show that the contributions of $B_k$ and $C_k$ are negligible. By monotonicity of $V$ and the induction assumption,
\begin{multline*}
B_k(t)+C_k(t)\leq V(h+t_0)(V_k(t)-V_k(t-t_0))+V(h)(V_k(t+h)-V_k(t))\sim (V(h+t_0)t_0+V(h)h)f_k(t)\\=o(f_{k+1}(t)),\quad t\to\infty.
\end{multline*}

Also, we note that by monotonicity of $f$
\begin{equation*}
\int_{(t-t_0,\,t]}f(t-x){\rm d}V_k(x)\leq f(t_0)(V_k(t)-V_k(t-t_0))=o(f_{k+1}(t)),\quad t\to\infty.
\end{equation*}
Now we pass to the analysis of the principal term $A_k$:
\begin{multline*}
A_k(t)\geq (h-\varepsilon)\Big(\int_{[0,\,t]}f(t-x){\rm d}V_k(x)-\int_{(t-t_0,\,t]}f(t-x){\rm d}V_k(x)\Big)=(h-\varepsilon) \int_{[0,\,t]}f(t-x){\rm d}V_k(x)\\+o(f_{k+1}(t))=:
(h-\varepsilon)(f\ast V_k)(t)+o(f_{k+1}(t)),\quad t\to\infty.
\end{multline*}
Put $$\varphi(s):=\int_{[0,\infty)}e^{-sx}{\rm d}f(x)\quad \text{and}\quad \psi_k(s):=\int_{[0,\infty)}e^{-sx}{\rm d}V_k(x),\quad s\geq 0.$$ By the implication (1.7.1) $\Rightarrow$ (1.7.2) of Theorem 1.7.1 in \cite{Bingham+Goldie+Teugels:1989}, $$\int_{[0,\,\infty)}e^{-sx}{\rm d}(f\ast V_k)(x)=\varphi(s)\psi_k(s)~\sim~ (\Gamma(\beta+1))^{k+1}s^{-(k+1)\beta-k} (\ell(1/s))^{k+1},\quad s\to 0+.$$ Invoking now the implication (1.7.2) $\Rightarrow$ (1.7.1) of the same theorem yields $$(f\ast V_k)(t)~\sim~ \frac{(\Gamma(\beta+1))^{k+1}}{\Gamma((k+1)(\beta+1))} t^{(k+1)\beta+k}(\ell(t))^{k+1}~\sim~ f_{k+1}(t),\quad t\to\infty.$$ Combining fragments together we arrive at  $${\lim\inf}_{t\to\infty}\frac{V_{k+1}(t+h)-V_{k+1}(t)}{f_{k+1}(t)}\geq h.$$ More precisely, we first obtain the last inequality with $h-\varepsilon$ on the right-hand side and then let $\varepsilon$ tend to $0+$. The proof of the converse inequality for the limit superior is analogous, hence omitted. This completes the proof of \eqref{eq:5ref}, hence of \eqref{eq:5}. Relation \eqref{eq:7} now follows from \eqref{eq:5} by another appeal to Theorem 3.7.3 in \cite{Bingham+Goldie+Teugels:1989}. The proof of Proposition \ref{prop:regular} is complete.
\end{proof}

Proposition \ref{prop:meanHaan} is of principal importance for what follows. It essentially states that whenever the function $\rho_j$ belongs to the de Haan class $\Pi$, so does $t\mapsto \me K^{(j)}_t$. We shall use the standard notation $x\vee y=\max(x,y)$ and $x\wedge y=\min(x,y)$ for $x,y\in\mr$.
\begin{assertion}\label{prop:meanHaan}
Under the assumptions of Theorem \ref{main}, for $\lambda>0$ and $j\in\mn$,
\begin{equation}\label{eq:4444}
\lim_{t\to\infty}\frac{\Phi_j(\lambda t)-\Phi_j(t)}{(\log t)^{j\beta+j-1}(\ell(\log t))^j}=\frac{(\Gamma(\beta+1))^j}{\Gamma(j(\beta+1))}\log\lambda,
\end{equation}
where $\Phi_j(t):=\me K^{(j)}_t=\sum_{|{\rr}|=j}(1-\eee^{- tp_{\rr}})$ for $t\geq 0$. In particular,
\begin{equation}\label{eq:33331}
\Phi^\prime_j(t)=\sum_{|{\rr}|=j}p_{\rr} \eee^{-t p_{\rr}}~\sim~ \frac{(\Gamma(\beta+1))^j}{\Gamma(j(\beta+1))} \frac{(\log t)^{j\beta+j-1}(\ell(\log t))^j}{t},\quad t\to\infty.
\end{equation}
\end{assertion}
\begin{proof}
We shall prove \eqref{eq:4444} in a form: for $\lambda>0$ and $j\in\mn$,
\begin{equation}\label{eq:33333}
\lim_{t\to\infty}\frac{ \Phi_j(\lambda t)-\Phi_j(t)}{g_j(t)}=\log \lambda,
\end{equation}
where $g_j(t):=f_j(\log t)$ for $t>0$, and $f_j$ is a nonnegative function satisfying \eqref{eq:5refref}. Assume first that $\lambda>1$. For any such an $\lambda$ and $t>0$, write
\begin{multline}
\Phi_j(\lambda t)-\Phi_j(t)=\sum_{|{\rr}|=j}\Big(\eee^{-t p_{\rr}}-\eee^{-\lambda t p_{\rr}}\Big)=\int_{(1,\infty)}\Big(\eee^{-t/x}-\eee^{-\lambda t/x}\Big){\rm d}\rho_j(x)\\=\int_1^\infty \Big(\eee^{-\lambda t/x}\frac{\lambda t}{x^2}-\eee^{-t/x}\frac{t}{x^2}\Big)\rho_j(x){\rm d}x\\=\int_0^t \eee^{-x}\Big(\rho_j\Big(\frac{\lambda t}{x}\Big)-\rho_j\Big(\frac{t}{x}\Big)\Big){\rm d}x+\int_{t}^{\lambda t}\eee^{-x}\rho_j\Big(\frac{\lambda t}{x}\Big){\rm d}x.\label{eq:covar}
\end{multline}
Here, the third equality is obtained with the help of integration by parts and the fact that $\lim_{x\to\infty}x^{-1}\rho_j(x)=0$ (according to Lemma 3 in \cite{Karlin:1967} this limit relation holds true for any counting function of probabilities; of course, in our setting it is also secured by the known asymptotics of $\rho_j$). The fourth equality results from the change of variables: $y=\lambda t/x$ for the first summand and $y=t/x$ for the second. Invoking \eqref{eq:covar} we infer, for $\lambda>1$, % and $t\geq 0$,
\begin{equation*}
\Phi_j(\lambda t)-\Phi_j(t)=\int_0^t \eee^{-x}\Big(\rho_j\Big(\frac{\lambda t}{x}\Big)-\rho_j\Big(\frac{t}{x}\Big)\Big){\rm d}x+o(1),\quad t\to\infty
\end{equation*}
because, by monotonicity of $\rho_j$,
\begin{equation*}
\int_t^{\lambda t} \eee^{-x}\rho_j\Big(\frac{\lambda t}{x}\Big){\rm d}x\leq \rho_j(\lambda) \int_t^{\lambda t}\eee^{-x}{\rm d}x~\to~0,\quad t\to\infty.
\end{equation*}
Thus, we have to show that
\begin{equation}
\lim_{t\to\infty}\frac{\int_0^t \eee^{-x}\Big(\rho_j\Big(\frac{\lambda t}{x}\Big)-\rho_j\Big(\frac{t}{x}\Big)\Big){\rm d}x}{g_j(t)}=\log \lambda.\label{eq:2}
\end{equation}
By Proposition \ref{prop:regular}, relation \eqref{eq:5} holds. This implies that given $\varepsilon>0$ there exists $t_1>0$ such that $$\frac{\rho_j(\lambda t/x)-\rho_j(t/x)}{g_j(t/x)}\leq \log \lambda+\varepsilon$$ whenever $x\in (0,t/t_1]$. Further, by Potter's bound (Theorem 1.5.6 (a) in \cite{Bingham+Goldie+Teugels:1989}), given $A>1$ and $\delta\in (0,1)$ there exists $t_2>0$ such that $$\frac{g_j(t/x)}{g_j(t)}\leq A (x^{-\delta}\vee x^\delta)$$ whenever $t\geq t_2$ and $x\in (0, t/t_2]$. Hence, $$\frac{\rho_j(\lambda t/x)-\rho_j(t/x)}{g_j(t)}\leq A(\log \lambda+\varepsilon)(x^{-\delta}\vee x^\delta)$$ whenever $x\in (0,t/t_0]$ and $t\geq t_0$, where $t_0:=t_1\vee t_2$. Since $\int_0^\infty \eee^{-x}(x^{-\delta}\vee x^\delta){\rm d}x<\infty$ and, by \eqref{eq:5}, for fixed $x>0$, $$\lim_{t\to\infty}\frac{\rho_j(\lambda t/x)-\rho_j(t/x)}{g_j(t)}=\log\lambda,$$ invoking the Lebesgue dominated convergence theorem yields
\begin{equation*}
\lim_{t\to\infty}\frac{\int_0^{t/t_0} \eee^{-x}\Big(\rho_j\Big(\frac{\lambda t}{x}\Big)-\rho_j\Big(\frac{t}{x}\Big)\Big){\rm d}x}{g_j(t)}=\log \lambda.
\end{equation*}
Noting that $$\int_{t/t_0}^t \eee^{-x}\Big(\rho_j\Big(\frac{\lambda t}{x}\Big)-\rho_j\Big(\frac{t}{x}\Big)\Big){\rm d}x\leq \rho_j(\lambda t_0) \int_{t/t_0}^t \eee^{-x}{\rm d}x~\to~0,\quad t\to\infty,$$ we arrive at \eqref{eq:2} which shows that \eqref{eq:33333} holds for $\lambda\geq 1$. Replacing in \eqref{eq:33333} $t$ with $t/\lambda$ ($\lambda>1$) and using the fact that $g_j$ is a slowly varying function we conclude that \eqref{eq:33333} also holds for $\lambda\in (0,1)$.

Since $\Phi_j^\prime$ is a nonincreasing function, Theorem 3.6.8 in \cite{Bingham+Goldie+Teugels:1989} in combination with \eqref{eq:33333} entail $$\Phi_j^\prime(t)~\sim~ \frac{g_j(t)}{t},\quad t\to\infty $$ which is equivalent to \eqref{eq:33331}.
\end{proof}

In Corollary \ref{prop:covariance} we identify the covariance of the limit process $Z$. In particular, the result suggests % demonstrates
that the limit process is stationary. We recall that the covariance of random variables $X$ and $Y$ with finite second moments is defined by ${\rm Cov}\,(X,Y)=\me XY-\me X\me Y$.
\begin{cor}\label{prop:covariance}
Under the assumptions of Theorem \ref{main}, for $u,v\in\mr$ and $j\in\mn$,
\begin{equation}\label{eq:414}
{\rm Var}\,K^{(j)}_{\eee^T}~\sim~\frac{(\log 2)(\Gamma(\beta+1))^j}{\Gamma(j(\beta+1))}T^{j\beta+j-1}(\ell(T))^j,\quad T\to\infty
\end{equation}
and
\begin{equation}\label{eq:3}
\lim_{T\to\infty}\frac{{\rm Cov}\,(K^{(j)}_{\eee^{T+u}}, K^{(j)}_{\eee^{T+v}})}{{\rm Var}\,K^{(j)}_{\eee^T}}=\frac{\log\big(1+e^{-|u-v|}\big)}{\log 2}.
\end{equation}
\end{cor}
\begin{proof}
We shall prove that
\begin{equation}\label{eq:3333}
\lim_{T\to\infty}\frac{{\rm Cov}\,(K^{(j)}_{\eee^{T+u}}, K^{(j)}_{\eee^{T+v}})}{f_j(T)}=\log\big(1+e^{-|u-v|}\big),
\end{equation}
where $f_j$ is a nonnegative function satisfying \eqref{eq:5refref}. Formula \eqref{eq:414} which states that
$${\rm Var}\,K^{(j)}_{\eee^T}~\sim~ (\log 2)f_j(T),\quad T\to\infty$$ is an immediate consequence of \eqref{eq:3333} with $u=v=0$. Formula \eqref{eq:3} is then implied by \eqref{eq:414} and \eqref{eq:3333}.

We start by showing that, for $s,t\geq 0$ and $j\in\mn$,
\begin{equation}
{\rm Cov}\,(K_s^{(j)}, K_t^{(j)})=%\me K_r^{(j)}K_s^{(j)}-\me K_r^{(j)}\me K_s^{(j)}
\sum_{|{\rr}|=j}\Big(\eee^{-(s\vee t)p_{\rr}}-\eee^{-(s+t)p_{\rr}}\Big)=\me K^{(j)}_{s+t}-\me K^{(j)}_{s\vee t}.\label{eq:covar112}
\end{equation}
Indeed, using \eqref{eq:jthgener} we obtain $$K_s^{(j)}K_t^{(j)}=\sum_{|{\rr}_1|=j}\1_{\{T_{\rr_1}\leq s\}}\sum_{|{\rr}_2|=j}\1_{\{T_{\rr_2}\leq t\}}=\sum_{|{\rr}|=j}\1_{\{T_{\rr}\leq s\wedge t\}}+\sum_{|{\rr}_1|=j}\1_{\{T_{\rr_1}\leq s\}}\sum_{|{\rr}_2|=j,\,\rr_2\neq \rr_1}\1_{\{T_{\rr_2}\leq t\}}.$$ Since the random variables $\1_{\{T_{\rr_1}\leq s\}}$ and $\sum_{|{\rr}_2|=j,\,\rr_2\neq \rr_1}\1_{\{T_{\rr_2}\leq t\}}=K_t^{(j)}-\1_{\{T_{\rr_1}\leq t\}}$ are independent we infer $${\rm Cov}\,(K_s^{(j)}, K_t^{(j)})=\sum_{|{\rr}|=j}(\mmp\{T_{\rr}\leq s\wedge t\}-\mmp\{T_{\rr}\leq s\}\mmp\{T_{\rr}\leq t\})=\sum_{|{\rr}|=j}\Big(\eee^{-(s\vee t)p_{\rr}}-\eee^{-(s+t)p_{\rr}}\Big).$$ The second equality in \eqref{eq:covar112} follows from $\me K^{(j)}_t=\sum_{|\rr|=j}(1-\eee^{-tp_{\rr}})$ for $t\geq 0$.

Putting in \eqref{eq:covar112} $s=\eee^{T+u}$ and $t=\eee^{T+v}$ for $u,v\in\mr$ and $T\geq 0$ and invoking Proposition \ref{prop:meanHaan} we infer, as $T\to\infty$,
\begin{equation*}
\frac{{\rm Cov}\,(K_{\eee^{T+u}}^{(j)}, K_{\eee^{T+v}}^{(j)})}{f_j(T)}=\frac{\me K_{\eee^{T}(\eee^u+\eee^v)}^{(j)}-\me K_{\eee^{T}(\eee^u\vee \eee^v)}^{(j)}}{f_j(T)}~\to~\log\Big(\frac{\eee^u+\eee^v}{\eee^u\vee \eee^v}\Big) =\log \big(1+\eee^{-|u-v|}\big). %,\quad T\to\infty.
\end{equation*}
\end{proof}

In Corollary \ref{prop:covariance2} we provide a crude asymptotics of ${\rm Cov}\,(K^{(i)}_{\eee^{T+u}}, K^{(j)}_{\eee^{T+v}})$ for $i\neq j$. This result is sufficient for one part of the proof of Theorem \ref{main}.
\begin{cor}\label{prop:covariance2}
Under the assumptions of Theorem \ref{main}, for $u,v\in\mr$ and $i,j\in\mn$, $i<j$,
\begin{equation}\label{eq:1234}
\lim_{T\to\infty}\frac{{\rm Cov}\,(K^{(i)}_{\eee^{T+u}}, K^{(j)}_{\eee^{T+v}})}{({\rm Var}\,K^{(i)}_{\eee^T}{\rm Var}\,K^{(j)}_{\eee^T})^{1/2}}=0.
\end{equation}
\end{cor}
\begin{proof}
We shall show that
\begin{equation}\label{eq:inter}
{\rm Cov}\,(K^{(i)}_{\eee^{T+u}}, K^{(j)}_{\eee^{T+v}})=O(f_i(T)),\quad T\to\infty,
\end{equation}
where $f_i$ is a nonnegative function satisfying \eqref{eq:5refref} with $i$ replacing $j$. Since $f_i(T)\sim {\rm Var}\,K^{(i)}_{\eee^T}/\log 2$ and, according to 
\eqref{eq:414}, ${\rm Var}\,K^{(i)}_{\eee^T}=o\big({\rm Var}\,K^{(j)}_{\eee^T}\big)$ as $T\to\infty$, we conclude that \eqref{eq:inter} ensures \eqref{eq:1234}.

Arguing as in the proof of Corollary \ref{prop:covariance} we obtain, for $s,t\geq 0$ and $i,j\in\mn$, $i<j$,
\begin{equation*}
{\rm Cov}\,(K_s^{(i)}, K_t^{(j)})= \sum_{|{\rr}_1|=i}\eee^{-sp_{\rr_1}}\sum_{|{\rr}_2|=j-i}\big(\eee^{-(t-s)_+p_{\rm r_1}p_{\rr_2}}-\eee^{-tp_{\rr_1}p_{\rr_2}}\big).
\end{equation*}
Further, $${\rm Cov}\,(K_s^{(i)}, K_t^{(j)})\leq \sum_{|{\rr}_1|=i}\eee^{-sp_{\rr_1}}\sum_{|{\rr}_2|=j-i}\big(1-\eee^{-tp_{\rr_1}p_{\rr_2}}\big)\leq t\sum_{|{\rr}|=i}p_{\rr} \eee^{-sp_{\rr}}=t\Phi_i^\prime(s),$$ where $\Phi_i(s)=\me K^{(i)}_s$ for $s\geq 0$. In view of \eqref{eq:33331}, $$\Phi_i^\prime(s)\sim s^{-1}f_i(\log s),\quad s\to\infty.$$ Putting $s=\eee^{T+u}$ and $t=\eee^{T+v}$ for $u,v\in\mr$ and then sending $T\to\infty$ we arrive at \eqref{eq:inter}.
\end{proof}

The two results given next are needed for the proof of tightness in Theorem \ref{main}. Corollary \ref{lem:de_haan_theorem} follows from Proposition \ref{prop:regular} and Corollary \ref{prop:covariance} with the help of an additional argument.
\begin{cor}\label{lem:de_haan_theorem}
Under the assumptions of Theorem \ref{main}, for $j\in\mn$ and $a>0$,
$$\sum_{|{\rr}|=j} \eee^{-a |t + \log p_{\rr}|}~\sim~ \frac{2}{a\log 2}{\rm Var}\,K^{(j)}_{\eee^t},\quad t\to \infty.
$$
\end{cor}
\begin{proof}
Observe that
\begin{align*}
\sum_{|{\rr}|=j} \eee^{-a |t + \log p_{\rr}|}&= \int_{(1,\infty)} \eee^{-a |t-\log x|} \dd \rho_j(x)
&= \int_{(1,\,\eee^t]} \eee^{-a (t-\log x)} \dd \rho_j(x) + \int_{(\eee^t,\,\infty)} \eee^{a (t-\log x)} \dd \rho_j(x).
\end{align*}
Integration by parts yields
\begin{equation}\label{eq:8}
\sum_{|{\rr}|=j} \eee^{-a|t + \log p_{\rr}|}= \Big(\rho_j(\eee^t) - \eee^{-at} \int_{(1,\,\eee^t]}ax^a \rho_j(x)  \frac{\dd x}{x}\Big)
+\Big(\eee^{at}\int_{(\eee^t,\,\infty)}a x^{-a} \rho_j(x)  \frac{\dd x}{x}-\rho_j(\eee^t)\Big).
\end{equation}
In view of \eqref{eq:5} and \eqref{eq:414}, for all $\lambda>0$,
\begin{equation*}
\lim_{t\to\infty}\frac{\rho_j(\lambda t)-\rho_j(t)}{(\log 2)^{-1}{\rm Var}\,K^{(j)}_t}=\log\lambda,
\end{equation*}
and the function $t\mapsto {\rm Var}\,K^{(j)}_t$ is slowly varying at $\infty$. Hence, by de Haan's theorem (Theorem 3.7.1 in~\cite{Bingham+Goldie+Teugels:1989}), each summand on the right-hand side of \eqref{eq:8} is asymptotically equivalent to $(a\log 2)^{-1}{\rm Var}\,K^{(j)}_{\eee^t}$ as $t\to\infty$.
\end{proof}

\begin{lemma}\label{lem:gumbel_density}
Fix some $A>0$. Let $G$ be a standard Gumbel random variable, that is, $\mmp\{G \leq x\} =\exp(- \eee^{-x})$ for % all
$x\in \R$. Then there is a constant $C=C(A)>0$ such that
$$\mmp\{s+u< -G \leq s+v\} \leq C (v-u) \, \eee^{-|s|}$$ for all $u<v$ from the interval $[-A,\,A]$ and all $s\in \R$.
\end{lemma}
\begin{proof}
The density of $-G$, namely, $g(x) = \eee^{x}\exp(-\eee^x)$ increases on the negative halfline and decreases on the positive halfline. Moreover, we can find $c_1>0$ such the inequality $g(x) \leq c_1 \eee^{-|x|}$ holds for $x\in\mr$. If $s>A$, then $0 < s + u < s+v$ and % we have
$$
\mmp\{s+u<-G \leq s+v\}\leq (v-u) g(s+u) \leq (v-u) g(s-A) \leq c_1 \eee^{A} (v-u) \eee^{-s}.
$$
If $s<-A$, then $s+u < s+v <0$ and a similar estimate holds true. % applies.
Finally, if $s\in [-A,\,A]$, then $s+u$ and $s+v$ are contained in the interval $[-2A,\,2A]$ and % we have
$$
\mmp\{s+u< -G \leq s+v\} \leq C (v-u),
$$
where $C:=\sup_{x\in [-2A,\,2A]} g(x)$.
\end{proof}

\section{Proof of Theorem \ref{main}}

At the first step we prove \eqref{relation_main5000} for one coordinate. At the second step we derive \eqref{relation_main5000} in full generality.

\noindent {\sc Step 1}. Fix $j\in\mn$. Thus, now we are focussing at the relation
\begin{equation}\label{relation_main500}
{\bf K}^{(j)}(T)~\Rightarrow~Z_j,\quad T\to\infty
\end{equation}
in the $J_1$-topology on $D$. The structure of the subsequent proof is standard: we prove weak convergence of finite-dimensional distributions and then check tightness.

According to the Cram\'{e}r-Wold device relation, weak convergence of finite-dimensional distributions is equivalent to the following limit relation
\begin{equation}    \label{eq:Cramer-Wold device}
\sum_{i=1}^k \alpha_i {\bf K}^{(j)}(T,u_i)~{\overset{{\rm d}}\longrightarrow}~ \sum_{i=1}^k \alpha_i Z_j(u_i),\quad T\to\infty
\end{equation}
for all $k\in\mn$, all $\alpha_1,\ldots, \alpha_k\in\mr$ and all
$-\infty<u_1<\ldots<u_k<\infty$. For $u\in\mr$, $T\geq 0$ and $\rr\in \mn^j$ (a box belonging to the $j$th generation), put $$\tilde B_{\rr}(T,u):=\1_{\{T_{\rr}\leq \eee^{T+u}\}}-\mmp\{T_{{\rr}}\leq \eee^{T+u}\}.$$ In view of \eqref{eq:jthgener} the left-hand side in \eqref{eq:Cramer-Wold device} is then equal to
$$\frac{\sum_{|{\rr}|=j}\sum_{i=1}^k \alpha_i\tilde B_{\rr}(T,u_i)}{({\rm Var}\,K^{(j)}_{\eee^T})^{1/2}}$$ and as such is the normalized (infinite) sum of independent centered random variables with finite second moments. Hence, in order to prove \eqref{eq:Cramer-Wold device}, it suffices to show (see, for instance, Theorem 3.4.5 on p.~129 in \cite{Durrett:2010}) that
\begin{align}       \label{eq:mgale CLT1}
\lim_{T\to\infty}\me\Big(\sum_{i=1}^k \alpha_i {\bf K}^{(j)}(T,u_i)\Big)^2=\me\Big(\sum_{i=1}^k \alpha_i Z_j(u_i)\Big)^2=\sum_{i=1}^k\alpha_i^2+2\sum_{1\leq i<\ell\leq k}\alpha_i  \alpha_\ell \frac{\log (1+\eee^{-(u_\ell-u_i)})}{\log 2}
\end{align}
and
\begin{equation}    \label{eq:mgale CLT2}
\lim_{T\to\infty} \sum_{|{\rr}|=j}\me\Big( \frac{(\sum_{i=1}^k \alpha_i \tilde B_{\rr}(T,u_i))^2}{{\rm Var}\,K^{(j)}_{\eee^T}}\1_{\{|\sum_{i=1}^k \alpha_i \tilde B_{\rr}(T,u_i)|>\varepsilon ({\rm Var}\,K^{(j)}_{\eee^T})^{1/2}\}}\Big)=0
\end{equation}
for all $\varepsilon>0$. Relation \eqref{eq:mgale CLT1} immediately follows from Corollary \ref{prop:covariance}. In view of the inequality
\begin{eqnarray}\label{eq:tech}
(a_1+\ldots+a_k)^2\1_{\{|a_1+\ldots+a_k|>y\}}&\leq&
(|a_1|+\ldots+|a_k|)^2\1_{\{|a_1|+\ldots+|a_k|>y\}}\notag\\&\leq&
k^2 (|a_1| \vee\ldots\vee |a_k|)^2\1_{\{k(|a_1| \vee\ldots\vee
|a_k|)>y\}}\notag\\&\leq&
k^2\big(a_1^2\1_{\{|a_1|>y/k\}}+\ldots+a_k^2\1_{\{|a_k|>y/k\}}\big).
\end{eqnarray}
which is valid for real $a_1,\ldots, a_m$ and $y>0$, relation \eqref{eq:mgale CLT2} is a consequence of
\begin{equation*}
\lim_{T\to\infty} \sum_{|{\rr}|=j}\me \Big( \frac{(\tilde B_{\rr}(T,u))^2}{{\rm Var}\,K^{(j)}_{\eee^T}}\1_{\{|\tilde B_{\rr}(T,u)|>\varepsilon ({\rm Var}\,K^{(j)}_{\eee^T})^{1/2}\}}\Big)=0,
\end{equation*}
where $u\in\R$ is fixed. The latter holds trivially, for $|\tilde B_{\rr}(T,u)|\leq 1$ a.s.\ whence the indicator $\1_{\{|\tilde B_{\rr}(T,u)|>\ldots\}}$ is equal to $0$ for large $T$. The proof of \eqref{eq:Cramer-Wold device} is complete.

Our next purpose is to prove that the family of laws of the stochastic processes $({\bf K}^{(j)}(T))_{T>0}$ is tight on the Skorokhod space $D[-A,\,A]$ for any fixed $A>0$. To this end, we shall show that there is a constant $C_1>0$ such that
\begin{equation}\label{eq:tight}
\me ({\bf K}^{(j)}(T,v)-{\bf K}^{(j)}(T,u))^2 ({\bf K}^{(j)}(T,w)-{\bf K}^{(j)}(T,v))^2\leq C_1 (w-u)^2
\end{equation}
for all $u<v<w$ in the interval $[-A,\,A]$ and large $T>0$. Together with the already proved fact that ${\bf K}^{(j)}(T,0)$ converges in distribution, this would imply the claimed tightness by a well-known sufficient condition (see Theorem 13.5 and formula (13.14) on p.~143 in~\cite{Billingsley:1999}).

For the box $\rr$ with $|{\rr}|=j$, introduce the following Bernoulli random variables
\begin{equation}\label{eq:B_k_C_k_def}
L_{\rr}:= \ind_{\{T+u < -\log p_{\rr} - G_{\rr} \leq T+v\}}, \qquad   M_{\rr}:= \ind_{\{T+v < -\log p_{\rr} - G_{\rr} \leq T+w\}}
\end{equation}
as well as their centered versions
$$
\widetilde L_{\rr}:= L_{\rr} - \E L_{\rr}, \qquad \widetilde M_{\rr} := M_{\rr} - \E M_{\rr}.
$$
We note that all these random variables depend on $u,v,w$ and $T$, the fact suppressed in our notation.
Let
$$
q_{\rr_1}:= \mmp\{L_{\rr_1} = 1\} = \E L_{\rr_1}, \qquad  z_{\rr_2}:= \mmp\{M_{\rr_2} = 1\} = \E M_{\rr_2}.
$$
Recalling~\eqref{eq:B_k_C_k_def} and using Lemma~\ref{lem:gumbel_density} we have
\begin{equation}\label{eq:q_k_est}
q_{\rr_1} \leq C  (v-u) \eee^{-|T+\log p_{\rr_1}|},
\qquad
z_{\rr_2} \leq C  (w-v) \eee^{-|T+\log p_{\rr_2}|},\quad \rr_1,\rr_2\in \mathbb{N}^j.
\end{equation}
In view of \eqref{eq:jthgener},
$$({\rm Var}\,K^{(j)}_{\eee^T})^{1/2} ({\bf K}^{(j)}(T,v)-{\bf K}^{(j)}(T,u))
= \sum_{|{\rr}_1|=j} \widetilde L_{\rr_1}$$ and $$({\rm Var}\,K^{(j)}_{\eee^T})^{1/2}({\bf K}^{(j)}(T,w)-{\bf K}^{(j)}(T,v))
= \sum_{|{\rr}_2|=j} \widetilde M_{\rr_2},$$
so that \eqref{eq:tight} is equivalent to
$$\E \Big(\sum_{|{\rr}_1|=j}\widetilde L_{\rr_1}\Big)^2 \Big(\sum_{|{\rr}_2|=j} \widetilde M_{\rr_2}\Big)^2 \leq C_1 (w-u)^2 \big({\rm Var}\,K^{(j)}_{\eee^T}\big)^2$$
for all $u<v<w$ in the interval $[-A,\,A]$ and large $T>0$. Multiplying the terms out, our task reduces to showing that
$$\sum_{\rr_1,\rr_2,\rr_3,\rr_4\in \N^j}\E \left[\widetilde L_{\rr_1}\widetilde L_{\rr_3}\widetilde M_{\rr_2}\widetilde M_{\rr_4} \right] \leq C_1 (w-u)^2 \big({\rm Var}\,K^{(j)}_{\eee^T}\big)^2.$$ If ${\rr}_1$ is not equal to any of the tuples ${\rr_2},{\rr_3},{\rr_4}$, then $\widetilde L_{\rr_1}$ is independent of the vector $(\widetilde L_{\rr_3}, \widetilde M_{\rr_2}, \widetilde M_{\rr_4})$, and we can take $\widetilde L_{\rr_1}$ out of the expectation implying that $\E [\widetilde L_{\rr_1}\widetilde L_{\rr_3}\widetilde M_{\rr_2}\widetilde M_{\rr_4}] = 0$. More generally,  if one of the tuples $\rr_1,\rr_2,\rr_3,\rr_4$ is not equal to any of the remaining ones, then the expectation vanishes. In the following, we shall consider collections $(\rr_1,\rr_2,\rr_3,\rr_4)$ in which every tuple is equal to some other tuple.

\vspace*{2mm}
\noindent
\textsc{Case 1.}
Consider first the case when ${\rr_1}\neq {\rr_3}$. Then, either ${\rr_2} ={\rr_1}$ and ${\rr_4} ={\rr_3}$, or ${\rr_2} ={\rr_3}$ and ${\rr_4} ={\rr_1}$. Consider the first case because the second one is similar.  The corresponding contribution is
$$
\sum_{\rr_1\neq \rr_3} \E \left[\widetilde L_{\rr_1}\widetilde L_{\rr_3}\widetilde M_{\rr_1}\widetilde M_{\rr_3}\right]
=
\sum_{\rr_1\neq \rr_3} \E \left[\widetilde L_{\rr_1}\widetilde M_{\rr_1}\right] \E \left[ \widetilde L_{\rr_3}\widetilde M_{\rr_3}\right].
$$
Since $L_{\rr_1}$ and $M_{\rr_1}$ cannot be equal to $1$ at the same time, $$\E \left[\widetilde L_{\rr_1}\widetilde M_{\rr_1}\right]
=-\me L_{\rr_1}\me M_{\rr_1}=-q_{\rr_1}z_{\rr_1}.$$
Analogously, $\E \left[\widetilde L_{\rr_3}\widetilde M_{\rr_3}\right]=-q_{\rr_3}z_{\rr_3}$. 
It follows that
$$
\sum_{\rr_1\neq \rr_3} \E \left[\widetilde L_{\rr_1}\widetilde L_{\rr_3}\widetilde M_{\rr_1}\widetilde M_{\rr_3}\right]=\sum_{\rr_1\neq \rr_3} q_{\rr_1}z_{\rr_1}q_{\rr_3}z_{\rr_3}\leq \sum_{|{\rr}_1|=j}q_{\rr_1}\sum_{|{\rr}_2|=j}z_{\rr_2}.
$$
Invoking \eqref{eq:q_k_est}, we arrive at
\begin{equation}\label{eq:est_sum_q_r}
\sum_{|{\rr}_1|=j}q_{\rr_1}\sum_{|{\rr}_2|=j}z_{\rr_2}\leq C^2 (w-u)^2 \Big(\sum_{|{\rr}|=j}\eee^{-|T+\log p_{\rr}|}\Big)^2\leq 8(\log 2)^{-2} C^2 (w-u)^2 \big({\rm Var}\,K^{(j)}_{\eee^T}\big)^2
\end{equation}
for all $u<v<w$ in the interval $[-A,\,A]$ and large $T>0$, where we have used Corollary \ref{lem:de_haan_theorem} for the last inequality.

\vspace*{2mm}
\noindent
\textsc{Case 2.}
Let now ${\rr}_1 =\rr_3$. Then we must also have $\rr_2=\rr_4$, for otherwise the expectation $\E [\widetilde L_{\rr_1}\widetilde L_{\rr_3}\widetilde M_{\rr_2}\widetilde M_{\rr_4}]$ vanishes.  The corresponding contribution can be estimated as follows:
\begin{multline*}
\sum_{r_1, r_2\in\mn^j}\E \left[\widetilde L_{\rr_1} \widetilde L_{\rr_1} \widetilde M_{\rr_2}\widetilde M_{\rr_2}\right]
=
\sum_{\rr_1\neq \rr_2} \E \left[\widetilde L_{\rr_1}^2 \right] \E \left[\widetilde M_{\rr_2}^2\right]
+
\sum_{|{\rr}|=j}\E \left[\widetilde L_{\rr}^2\widetilde M_{\rr}^2\right]
\\
\leq \sum_{\rr_1\neq \rr_2} q_{\rr_1}z_{\rr_2}+ 2\sum_{|{\rr}|=j} q_{\rr}z_{\rr}\leq 2 \sum_{|{\rr}_1|=j}q_{\rr_1}\sum_{|{\rr}_2|=j}z_{\rr_2}\leq 16(\log 2)^{-2} C^2 (w-u)^2 \big({\rm Var}\,K^{(j)}_{\eee^T}\big)^2
\end{multline*}
for all $u<v<w$ in the interval $[-A,\,A]$ and large $T>0$, by~\eqref{eq:est_sum_q_r}. Here, we have used the inequalities $\E [\widetilde L_{\rr_1}^2]=q_{\rr_1}(1-q_{\rr_1})\leq q_{\rr_1}$, $\E [\widetilde M_{\rr_2}^2]=z_{\rr_2}(1-z_{\rr_2})\leq z_{\rr_2}$ and
\begin{multline*}
\E \left[\widetilde L_{\rr}^2  \widetilde M_{\rr}^2\right]=q_{\rr} (1-q_{\rr})^2(-z_{\rr})^2 +z_{\rr}(1-z_{\rr})^2 (-q_{\rr})^2 + (1-q_{\rr}-z_{\rr})(-q_{\rr})^2(-z_{\rr})^2\\=q_{\rr}z_{\rr}(q_{\rr}+z_{\rr}-3q_{\rr}z_{\rr})
\leq 2 q_{\rr}z_{\rr}.
\end{multline*}
The first equality stems from the fact that $L_{\rr}$ and $M_{\rr}$ cannot be equal to $1$ simultaneously.

\noindent {\sc Step 2}. In view of the already proved tightness of the families of laws of coordinates on the left-hand side of \eqref{relation_main5000}, the family of laws of the stochastic processes $\big({\bf K}^{(j)}(t)\big)_{j\geq 1}$, $t>0$ is tight on $D^\mn$ equipped with the product $J_1$ topology. Thus, according to the Cram\'{e}r-Wold device relation, it remains to prove convergence in distribution of the linear combinations of the coordinates on the left-hand side of \eqref{relation_main5000} to the corresponding linear combinations of the coordinates on the right-hand side of \eqref{relation_main5000}. To the end, we first show that, for all $k\in\mn$, all $\alpha_1,\ldots, \alpha_k\in\mr$ and all $-\infty<u_1<\ldots<u_k<\infty$,
\begin{equation}\label{eq:finconv}
\sum_{j=1}^k \alpha_j {\bf K}^{(j)}(T,u_j)~{\overset{{\rm d}}\longrightarrow}~ \sum_{j=1}^k \alpha_i Z_j(u_j),\quad T\to\infty.
\end{equation}
Observe that the left-hand side is the infinite sum of independent centered random variable with finite second moments as is seen from a representation
\begin{multline*}
\sum_{j=1}^k \alpha_j {\bf K}^{(j)}(T,u_j)\\=\sum_{|{\rr}_1|=1}\Big(\frac{\alpha_1 \tilde B_{{\rr}_1}(T,u_1)}{({\rm Var}\,K^{(1)}_{\eee^T})^{1/2}}+\frac{\alpha_2\sum_{|{\rr}_2|=1}\tilde B_{{\rr}_1{\rr}_2}(T,u_2)}{({\rm Var}\,K^{(2)}_{\eee^T})^{1/2}}+\ldots+\frac{\alpha_k\sum_{|{\rr}_k|=k-1}\tilde B_{{\rr}_1{\rr}_k}(T,u_k)}{({\rm Var}\,K^{(k)}_{\eee^T})^{1/2}} \Big),
\end{multline*}
where an individual ${\rr}_1{\rr}_i$ with $|{\rr}_i|=i-1$ is a successor of ${\rr}_1$ in the $i$th generation. Note that, for the given ${\rr}_1$, the variables $\tilde B_{{\rr}_1}(T,u_1)$, $\sum_{|{\rr}_2|=1}\tilde B_{{\rr}_1{\rr}_2}(T,u_2),\ldots$ are dependent, yet the terms of the series (which correspond to different ${\rr}_1$) are independent.

To prove \eqref{eq:finconv}, we have to show that
\begin{align}\label{eq:limi}
\lim_{T\to\infty}\me\Big(\sum_{j=1}^k \alpha_j {\bf K}^{(j)}(T,u_j)\Big)^2=\me\Big(\sum_{j=1}^k \alpha_j Z_j(u_j)\Big)^2=\sum_{j=1}^k\alpha_i^2
\end{align}
and that the Lindeberg condition (a counterpart of \eqref{eq:mgale CLT2}) holds. Formula \eqref{eq:limi} is secured by Corollary \ref{prop:covariance2}. In view of \eqref{eq:tech} and the proof for Step 1 the Lindeberg condition follows if we can check that, for all $\varepsilon>0$ and all $j\in\mn$,
\begin{equation}\label{eq:linde}
\lim_{T\to\infty} \sum_{|{\rr}_1|=1} \me \Big(\big({\rm Var}\,K^{(j)}_{\eee^T}\big)^{-1}\Big(\sum_{|{\rr}_2|=j-1}\tilde B_{{\rr}_1{\rr}_2}(T,\,0)\Big)^2\1_{\{|\sum_{|{\rr}_2|=j-1} \tilde B_{{\rr}_1{\rr}_2}(T,\,0)|>\varepsilon ({\rm Var}\,K^{(j)}_{\eee^T})^{1/2}\}}\Big)=0.
\end{equation}
Here, a possibility of investigating $\tilde B_{{\rr}_1{\rr}_2}(T,\,0)$ in place of $\tilde B_{{\rr}_1{\rr}_2}(T,\,u)$ for $u\in\mr$ is justified by the regular variation of the function $T\mapsto {\rm Var}\,K^{(j)}_{\eee^T}$, see \eqref{eq:414}. We shall prove that
\begin{equation*}
\lim_{T\to\infty} \big({\rm Var}\,K^{(j)}_{\eee^T}\big)^{-2} \sum_{|{\rr}_1|=1} \me \Big(\sum_{|{\rr}_2|=j-1}\tilde B_{{\rr}_1{\rr}_2}(T,\,0)\Big)^4=0
\end{equation*}
which obviously entails \eqref{eq:linde}. To this end, put $a_{{\rr}_1{\rr}_2}(T):=1-\exp(-\eee^T p_{{\rr}_1{\rr}_2})$ for ${\rr}_1, {\rr_2}\in \mathcal{R}$ and $T>0$ and write
\begin{multline*}
\me \Big(\sum_{|{\rr}_2|=j-1}\tilde B_{{\rr}_1{\rr}_2}(T,\,0)\Big)^4=\sum_{|{\rr}_2|=j-1} \me \big(\tilde B_{{\rr}_1{\rr}_2}(T,\,0))^4\\+3\sum_{|{\rr}_2|=j-1, |{\rr}_3|=j-1, {\rr_2}\neq {\rr}_3}\me (\tilde B_{{\rr}_1{\rr}_2}(T,\,0))^2\me (\tilde B_{{\rr}_1{\rr}_3}(T,\,0))^2 \\=\sum_{|{\rr}_2|=j-1}(1-a_{{\rr}_1{\rr}_2}(T))(a_{{\rr}_1{\rr}_2}(T)-3a^2_{{\rr}_1{\rr}_2}(T)(1-a_{{\rr}_1{\rr}_2}(T)))\\+3\sum_{|{\rr}_2|=j-1, |{\rr}_3|=j-1, {\rr_2}\neq {\rr}_3}a_{{\rr}_1{\rr}_2}(T)(1-a_{{\rr}_1{\rr}_2}(T))a_{{\rr}_1 {\rr}_3}(T)(1-a_{{\rr}_1{\rr}_3}(T))\\\leq \sum_{|{\rr}_2|=j-1}(1-a_{{\rr}_1{\rr}_2}(T))a_{{\rr}_1{\rr}_2}(T)
+3\Big(\sum_{|{\rr}_2|=j-1}(1-a_{{\rr}_1{\rr}_2}(T))a_{{\rr}_1{\rr}_2}(T)\Big)^2={\rm Var}\,K^{(j-1)}_{\eee^T p_{{\rr}_1}}+3({\rm Var}\,K^{(j-1)}_{\eee^T p_{{\rr}_1}})^2\\\leq (1+3 \me K^{(j-1)}_{\eee^T})  {\rm Var}\,K^{(j-1)}_{\eee^T p_{{\rr}_1}}
\end{multline*}
having utilized $\sum_{|{\rr}_2|=j-1}(1-a_{{\rr}_1{\rr}_2}(T))a_{{\rr}_1{\rr}_2}(T)\leq \sum_{|{\rr}_2|=j-1}a_{{\rr}_1{\rr}_2}(T)=\me K^{(j-1)}_{\eee^Tp_{{\rr}_1}}$ and monotonicity of $t\mapsto \me K_t^{(j-1)}$ for the last inequality.
Hence, $$\sum_{|{\rr}_1|=1} \me \Big(\sum_{|{\rr}_2|=j-1}\tilde B_{{\rr}_1{\rr}_2}(T,\,0)\Big)^4\leq (1+3 \me K^{(j-1)}_{\eee^T})  {\rm Var}\,K^{(j)}_{\eee^T}=o\big(\big({\rm Var}\,K^{(j)}_{\eee^T}\big)^2\big),\quad T\to\infty.$$ To explain the last equality, we recall that, according to \eqref{eq:414},
\begin{equation*}
{\rm Var}\,K^{(j)}_{\eee^T}~\sim~ {\rm const}\, T^{j\beta+j-1}(\ell(T))^j,\quad T\to\infty.
\end{equation*}
Further, a combination of \eqref{eq:7} and Theorem 1 in \cite{Karlin:1967} yields
$$\me K^{(j-1)}_{\eee^T}~\sim~ \rho_{j-1}(\eee^T)~\sim~ {\rm const}\, T^{(j-1)(\beta+1)} (\ell (T))^{j-1},\quad T\to\infty.$$ It remains to note that $(j-1)(\beta+1)<j\beta+j-1$ when $\beta>0$ and that, by assumption, $\lim_{T\to\infty}\ell(T)=\infty$ when $\beta=0$. This finishes the proof of \eqref{eq:linde}, hence of \eqref{eq:finconv}. To complete the proof of Theorem \ref{main}, combine now \eqref{eq:finconv} with the arguments given at Step 1.

\begin{rem}
While proving some functional limit theorems with discontinuous converging processes $(X(T,u))$, say and continuous weak limits it may be sufficient to check tightness in the space of continuous functions on $[0,\infty)$. The simplest (if applicable) sufficient condition ensuring such a tightness is: there exist constants $\gamma>0$ and $\delta>1$ and a nondecreasing continuous function $F$ such that the inequality $$\me |X(T,u)-X(T,v)|^\gamma~\leq~ |F(u)-F(v)|^\delta$$ holds for all $u,v\geq 0$ and large $T$.

The latter inequality does not hold for the process ${\bf K}^{(j)}(T)$. Indeed, with $T$ fixed, $v=0$ and $u\to 0+$,
$$\me |{\bf K}^{(j)}(T,u)-{\bf K}^{(j)}(T,0)|^\gamma~\sim~ C_Tu$$ for a constant $C_T$. To check this, note that in the chosen setting the variable $K^{(j)}_{\eee^{T+u}}-K^{(j)}_{\eee^T}$ is very close to $\pi(\eee^{T+u})-\pi(\eee^T)$, for which explicit calculations are possible.
\end{rem}

\section{Proof of Corollary \ref{main1}}

Plainly, for $j,n\in\mn$ and $t\geq 0$, $K^{(j)}_t=\mathcal{K}^{(j)}_{\pi(t)}$ and conversely
\begin{equation}\label{eq:impo}
\mathcal{K}^{(j)}_n=K^{(j)}_{S_n}.
\end{equation}
Here, in the first equality the random variable $\pi(t)$ is independent of $(\mathcal{K}^{(j)}_n)_{n\in\mn}$, whereas in the second equality $S_n$ and $(K^{(j)}_t)_{t\geq 0}$ are dependent.

Dini's theorem implies that the strong law of large numbers for standard random walks with finite mean has a uniform version. For the particular standard random walk $(S_n)_{n\in\mn}$ as in \eqref{eq:1} it reads: for all $a,b\in\mr$, $a<b$,
\begin{equation}\label{132}
\lim_{T\to\infty}\sup_{a\leq u\leq b}|{\bf S}(T,u)-\psi(u)|=0\quad\text{a.s.},
\end{equation}
where, for $T\geq 0$ and $u\in\mr$, ${\bf S}(T,u):=\eee^{-T}S_{\lfloor \eee^{T+u}\rfloor}$ and $\psi(u)=\eee^u$. This in combination with \eqref{relation_main5000} gives
\begin{equation}\label{relation_main500021}
\big(\big({\bf K}^{(j)}(T)\big)_{j\geq 1}, {\bf S}(T)\big)~\Rightarrow~\big((Z_j)_{j\geq 1}, \psi\big)\quad T\to\infty
\end{equation}
in the product $J_1$-topology on $D^{\mn}\times D$, where ${\bf S}(T):=({\bf S}(T,u))_{u\in\mr}$.

It is known (see, for instance, Lemma 2.3 on p.~159 in
\cite{Gut:2009}) that, for fixed $j\in\mn$, the composition
mapping $((x_1,\ldots, x_j), \varphi)\mapsto (x_1\circ
\varphi,\ldots, x_j\circ \varphi)$ is continuous at vectors
$(x_1,\ldots, x_j): \mr^j\to \mr^j$ with continuous coordinates
and nondecreasing continuous $\varphi: \mr\to \mr_+$. Since $Z_j$ is a.s.\ continuous (see Theorem \ref{thm:main2}) and $\psi$ is
nonnegative, nondecreasing and continuous, we can invoke the continuous mapping theorem to infer
\begin{equation*}
\Big(\Big(\frac{\mathcal{K}^{(j)}_{\lfloor \eee^{T+u}\rfloor }- \Phi_j(S_{\lfloor \eee^{T+u}\rfloor})}{(\Psi_j(\eee^T))^{1/2}}\Big)_{u\in\mr}\Big)_{j\geq 1}~\Rightarrow~((Z_j(u))_{u\in\mr})_{j\geq 1},\quad T\to\infty
\end{equation*}
in the product $J_1$-topology on $D^{\mn}$. Here, $\Phi_j(t)=\me K^{(j)}_t$, $\Psi_j(t)={\rm Var}\,K^{(j)}_t$ for $t\geq 0$ and we have used \eqref{eq:impo}.

According to \eqref{eq:414}, the function $t\mapsto \Psi_j(\eee^t)$ is regularly varying at $\infty$, whence $\Psi_j(e^{T})~\sim~\Psi_j(\lfloor e^{T}\rfloor)$ as $T\to\infty$. Furthermore, by Lemma 4 in \cite{Gnedin+Hansen+Pitman:2007}, $\Psi_j(\lfloor \eee^T\rfloor)~\sim~{\rm Var}\,\mathcal{K}^{(j)}_{\lfloor \eee^T\rfloor}$ whenever $\lim_{T\to\infty}\Psi_j(T)=\infty$. Summarizing, ${\rm Var}\,\mathcal{K}^{(j)}_{\lfloor \eee^T\rfloor}~\sim~\Psi_j(\eee^T)$ as $T\to\infty$.

Thus, we are left with showing that, for all $a,b>0$, $a<b$
\begin{equation}\label{eq:cond1}
(\Psi_j(t))^{-1/2}\sup_{v\in [a,\,b]}|\Phi_j(S_{\lfloor tv\rfloor})-\Phi_j(tv)|~\overset{\mmp}{\to}~ 0,\quad t\to\infty
\end{equation}
and
\begin{equation}\label{eq:cond2}
\lim_{t\to\infty} (\Psi_j(t))^{-1/2}\sup_{v\in [a,\,b]}|\Phi_j(tv)-\me \mathcal{K}^{(j)}_{\lfloor tv\rfloor}|= 0.
\end{equation}
Here, for notational simplicity we have replaced $\eee^T$ with $t$ and $\eee^u$ with $v$.

\noindent {\sc Proof of \eqref{eq:cond1}}. Fix $a,b>0$, $a<b$ and put $\eta(at):=at\wedge S_{\lfloor at\rfloor}$ for $t\geq 0$. Write, for $x\geq 0$,
\begin{multline*}
\sup_{v\in[a,\,b]}\,|\Phi_j(S_{\lfloor tv\rfloor})-\Phi_j(tv)|\leq \sup_{v\in[a,\,b]}\,\sum_{|{\rr}|=j}\eee^{-(tv\wedge S_{\lfloor tv\rfloor})p_{\rr}}\big(1-\eee^{-|S_{\lfloor tv\rfloor}-tv|p_{\rr}}\big)\\\leq \sum_{|{\rr}|=j} \eee^{-\eta(at)p_{\rr}}\big(1-\eee^{-\sup_{v\in[a,\,b]}|S_{\lfloor tv\rfloor}-tv|p_{\rr}}\big)=\big(\Phi_j\big(\eta(at)+\sup_{v\in[a,\,b]}|S_{\lfloor tv\rfloor}-tv|\big)-\Phi_j(\eta(at))\big)\\\times (\1_{\{\sup_{v\in[a,\,b]}|S_{\lfloor tv\rfloor}-tv|\leq t^{1/2}x\}}+\1_{\{\sup_{v\in[a,\,b]}|S_{\lfloor tv\rfloor}-tv|> t^{1/2}x\}})=:A_j(t,x)+B_j(t,x)
\end{multline*}
having utilized monotonicity of $y\mapsto \eee^{-yp_{\rr}}$ for the second inequality.

We intend to prove that
\begin{equation}\label{eq:17}
\lim_{t\to\infty}\frac{A_j(t,x)}{(\Psi_j(t))^{1/2}}=0\quad\text{a.s.}
\end{equation}
To this end, we first note that $$A_j(t,x)\leq \Phi_j(\eta(at)+t^{1/2}x)-\Phi_j(\eta(at))\quad \text{a.s.}$$ It is shown in the proof of Lemma 4 in \cite{Gnedin+Hansen+Pitman:2007} that under the sole assumption that $\lim_{t\to\infty}\Psi_j(t)=\infty$, the function $\Phi_j^\prime(t)=\sum_{|{\rr}|=j}p_{\rr}\eee^{-tp_{\rr}}$ satisfies $$\lim_{t\to\infty}\frac{(\Phi_j^\prime(t))^2}{\Phi_j^\prime(2t)}=0.$$ With this at hand, arguing as in the proof of formula (13) in \cite{Gnedin+Hansen+Pitman:2007} we obtain a.s.
\begin{equation}\label{eq:15}
\frac{\Phi_j(\eta(at)+t^{1/2}x)-\Phi_j(\eta(at))}{(\Phi_j(2\eta(at))-\Phi_j(\eta(at)))^{1/2}}=\frac{\int_{\eta(at)}^{\eta(at)+t^{1/2}x}\Phi_j^\prime(y){\rm d}y}{(\int_{\eta(at)}^{2\eta(at)}\Phi_j^\prime(y){\rm d}y)^{1/2}}\leq \frac{t^{1/2}x \Phi_j^\prime(\eta(at))}{(\eta(at)\Phi_j^\prime(2\eta(at)))^{1/2}}\to 0,\quad t\to\infty. 
\end{equation}
Here, we have used the limit relation
\begin{equation}\label{eq:16}
\lim_{t\to\infty}\frac{\eta(at)}{t}=a\quad\text{a.s.}
\end{equation}
which follows from the strong law of large numbers for standard random walks. In view of the equality $\Psi_j(\eta(at))=\Phi_j(2\eta(at))-\Phi_j(\eta(at))$ and the fact that the function $\Psi_j$ is slowly varying at $\infty$ we conclude with the help of \eqref{eq:16} and the uniform convergence theorem for slowly varying functions (Theorem 1.2.1 in \cite{Bingham+Goldie+Teugels:1989}) that $$\lim_{t\to\infty}\frac{\Psi_j(\eta(at))}{\Psi_j(t)}=1\quad \text{a.s.}$$ This in combination with \eqref{eq:15} proves \eqref{eq:17}.

Before we proceed recall one known corollary to Donsker's theorem $$t^{-1/2}\sup_{v\in[a,\,b]}\,|S_{\lfloor tv\rfloor}-tv|  ~{\overset{{\rm
d}}\longrightarrow}~ \sup_{v\in [a,\,b]}\,|B(v)|,\quad t\to\infty,$$ where $(B(v))_{v\geq 0}$ is a standard Brownian motion. Using this and \eqref{eq:17} we write, for any $\varepsilon>0$ and any $x>0$,
\begin{multline*}
\mmp\big\{A_j(t,x)+B_j(t,x)>\varepsilon (\Psi_j(t))^{1/2}\big\}\leq \mmp\{A_j(t)>2^{-1}\varepsilon (\Psi_j(t))^{1/2}\}+\mmp\{B_j(t)>2^{-1}\varepsilon (\Psi_j(t))^{1/2}\}\\\leq o(1)+\mmp\{\sup_{v\in[a,\,b]}|S_{\lfloor tv\rfloor}-tv|> t^{1/2}x\}\to \mmp\{\sup_{v\in[a,\,b]}|B(v)|>x\},\quad t\to\infty.
\end{multline*}
Since the left-hand side does not depend on $x$ we obtain on letting $x\to\infty$ $$\frac{A_j(t,x)+B_j(t,x)}{(\Psi_j(t))^{1/2}}~\overset{\mmp}{\to}~ 0,\quad t\to\infty$$ and thereupon \eqref{eq:cond1}. 

\noindent {\sc Proof of \eqref{eq:cond2}}. It is shown in the proof of Lemma 4.2 in \cite{Durieu+Wang:2016} that, for $v>0$ and large $t$,
$$|\Phi_j(tv)-\me \mathcal{K}^{(j)}_{\lfloor tv\rfloor}|\leq 1+\frac{\Phi_j(\lfloor tv\rfloor)}{\lfloor tv\rfloor}.$$ Consequently, by monotonicity,
$$\sup_{v\in [a,\,b]}|\Phi_j(tv)-\me \mathcal{K}^{(j)}_{\lfloor tv\rfloor}|\leq 1+\frac{\Phi_j(\lfloor tb\rfloor)}{\lfloor ta\rfloor}\leq 1+\frac{\lfloor tb\rfloor}{\lfloor ta\rfloor},$$ and \eqref{eq:cond2} follows.

\vspace{5mm}

\noindent {\bf Acknowledgement}. AI and VK acknowledge support by the National Research Foundation of Ukraine (project 2020.02/0014 ``Asymptotic regimes of perturbed random walks: on the edge of modern and classical probability''). ZK acknowledges support by the German Research Foundation under Germany's Excellence Strategy  EXC 2044 -- 390685587, Mathematics M\"unster: Dynamics - Geometry - Structure.

\end{document}